\documentclass{amsart}
\makeatletter
\@namedef{subjclassname@2020}{%
  \textup{2020} Mathematics Subject Classification}
\makeatother
\usepackage{latexsym}
\usepackage{amssymb}
\usepackage{graphicx}
\usepackage{wrapfig}
\usepackage{amsthm}
\usepackage{float}
\usepackage{epsfig}
\usepackage{multirow}
\usepackage{xcolor}
\usepackage{caption}
\usepackage[toc,page]{appendix}
\usepackage{old-arrows} 
\usepackage{caption}
\usepackage{subcaption}
\usepackage{lmodern, mathtools}
\usepackage{mathtools}
\usepackage{hyperref}
\usepackage{marginnote}
\RequirePackage{color}
\usepackage{tikz}
\usepackage{amscd,enumerate}
\usepackage{amsfonts}
\hypersetup{ colorlinks,
linkcolor=blue,
filecolor=green,
urlcolor=blue,
citecolor=blue }
\usepackage{helvet}
\usepackage{bigstrut}
\usepackage{float, caption, booktabs, cellspace}
\usepackage{calrsfs}

\DeclareMathAlphabet{\pazocal}{OMS}{zplm}{m}{n}

\def\pl{\mathop{\hbox{\rm P}_l}}
\def\pc{\mathop{\hbox{\rm P}_c}}
\def\supp{\mathop{\hbox{\rm Supp}}}
\def\bis{\mathop{\hbox{\rm Bis}_{\hbox{\rm c}}}}

\usepackage{comment}

\makeatletter
\@addtoreset{subfigure}{framenumber}
\makeatother

\usepackage{etoolbox}
\AtBeginEnvironment{figure}{\setcounter{subfigure}{0}}

\newtheorem{lemma}{Lemma}[section]

\input xygraph
\input xy
\xyoption{all}

\usepackage{array}
\newcolumntype{L}[1]{>{\raggedright\let\newline\\\arraybackslash\hspace{0pt}}m{#1}}
\newcolumntype{C}[1]{>{\centering\let\newline\\\arraybackslash\hspace{0pt}}m{#1}}
\newcolumntype{R}[1]{>{\raggedleft\let\newline\\\arraybackslash\hspace{0pt}}m{#1}}

\usepackage{multicol}

\def\l{\lambda}
\def\a{\alpha}

\def\g{\gamma}

\def\m{\mu}

\def\remove#1{}

\def\geo{\G_E^{(0)}}
\def\path{\mathop{\hbox{\rm Path}}}
\newcommand{\G}{{\mathcal{G}}}

\newcommand{\Z}{{\mathbb{Z}}}

\def\im{\mathop{\hbox{\rm Im}}}

\usepackage[many]{tcolorbox}
\newtheorem{corollary}[lemma]{Corollary}
\newtheorem{theorem}[lemma]{Theorem}
\newtheorem{proposition}[lemma]{Proposition}
\newtheorem*{proposition*}{Proposition}
\newtheorem*{corollary*}{Corollary}
\newtheorem*{theorem*}{Theorem}
\newtheorem{remark}[lemma]{Remark}

\theoremstyle{definition}

\newtheorem{example}[lemma]{Example}

\newcommand{\ZZ}{\mathbb{Z}}

\newcommand{\0}{^{(0)}}

\DeclareMathOperator{\soc}{soc}

\DeclareMathOperator{\Span}{span}

\DeclareMathOperator{\Supp}{Supp}

\title{ On the socle of a class of  Steinberg algebras}
\color{black}
\author[L. Clark]{Lisa Orloff Clark}
\address{L.O. Clark: Victoria University of Wellington, Wellington, New Zealand}
\email{lisa.orloffclark@vuw.ac.nz}

\author[C. Gil]{Crist\'obal Gil Canto}
\address{C. Gil Canto: Departamento de Matem\'atica Aplicada, E.T.S. Ingenier\'\i a Inform\'atica, Universidad de M\'alaga, 
M\'alaga,   Spain.}
\email{cgilc@uma.es}

\author[D. Mart\'{\i}n]{Dolores Mart\'{\i}n Barquero}
\author[C. Mart\'{\i}n]{C\'andido Mart\'{\i}n Gonz\'alez}

\address{D. Martín Barquero: Departamento de Matem\'atica Aplicada, Escuela de Ingenier\'\i as Industriales, Universidad de M\'alaga, Campus de Teatinos s/n. 29071 M\'alaga.   Spain.}
\email{dmartin@uma.es}

\address{C. Martín González: Departamento de \'Algebra Geometr\'{\i}a y Topolog\'{\i}a, Fa\-cultad de Ciencias, Universidad de M\'alaga, Campus de Teatinos s/n. 29071 M\'alaga. Spain.} \email{candido\_m@uma.es}

\author[I. Ruiz]{Iv\'an Ruiz Campos}
\address{I. Ruiz Campos:  Departamento de \'Algebra Geometr\'{\i}a y Topolog\'{\i}a, Fa\-cultad de Ciencias, Universidad de M\'alaga, 
M\'alaga, Spain.}
\email{ivaruicam@uma.es}

\thanks{ The authors are supported by the Spanish Ministerio de Ciencia e Innovaci\'on through projects  PID2019-104236GB-I00/AEI/10.13039/501100011033 and PID2023-152673NB-I00 and by the Junta de Andaluc\'{\i}a  through project  FQM-336,  all of them with FEDER funds. The fourth author is supported by a Junta de Andalucía PID fellowship no. PREDOC\_00029. The third and fourth authors are partially funded by grant Fortalece 2023/03 of "Comunidad Autónoma de La Rioja".}

\begin{document}
\begin{abstract}
We study minimal left ideals in Steinberg algebras of Hausdorff groupoids.  We establish a relationship between minimal left ideals in the algebra and open singletons in the unit space of the groupoid. We apply this to obtain results about the socle of Steinberg algebras under certain hypotheses. This encompasses known results about Leavitt path algebras and improves on Kumjian-Pask algebra results to include higher-rank graphs that are not row-finite.
\end{abstract}

\maketitle

\color{black}
\section{Introduction} 
Steinberg algebras, introduced in \cite{Steinberg1} and independently in \cite{CFST}, are algebras associated to ample groupoids. The main purpose of this paper is to lay the foundation for a socle theory for Steinberg algebras of Hausdorff groupoids that generalises the socle theory of the Leavitt path algebras \cite{socletheory,pino2008socle} and the socle theory of Kumjian-Pask algebras \cite{kpsoc}. We take inspiration from these graph algebra papers and also from results about the group algebra of a finite group, since the three classes of algebras are examples of Steinberg algebras. 

The study of the socle of a group algebra is an intricate issue, with no definitive answer, only partial results have been given. That is why in this work we specialize to a case generalizing both Leavitt path and Kumjian-Pask algebras.

After a preliminaries section, the first step in our analysis is to study minimal left ideals, which we do in Section~\eqref{sec:leftideals}. Following \cite[Theorem 3.4]{pino2008socle}, the existence of minimal left (resp. right) ideals, corresponds with the existence of special elements in the graph. In more general Steinberg algebras, we show that minimal left (resp. right) ideals correspond with particular elements in the groupoid, that is,  elements in the unit space such that their isotropy group is finite (trivial in most of the cases we are interested in) and such that the singleton set containing each one of them is open, see Proposition~\ref{prop_existence_singleton} and Theorem~\ref{puerto}. As an application, we apply our results to Kumjian-Pask algebras of finitely aligned $k$-graphs and generalise \cite[Theorem~3.9]{kpsoc} for row-finite $k$-graphs.  Finally, in Section~\eqref{sec:socle}, we describe the socle of an arbitrary Steinberg algebra, proving our main theorem in 
Theorem~\ref{thm:main}.

\section{Basic definitions and preliminary results}

\subsection{Steinberg algebras}
A {\it groupoid} $\G$ is a generalisation of a group in which the ‘binary operation’ is only partially defined (see \cite{Renault} for definition). 
Let $\G$ be a groupoid and $R$ a commutative (unital) ring. The {\it unit space} of $\G$ is the set 
\[\{xx^{-1} \colon x\in \G\}=\{x^{-1}x \colon x\in \G\}\] formed by  its idempotent elements  that we denote $\G^{(0)}$. We define $r,s \colon \G \to \G^{(0)}$ by $r(x)=x x^{-1}$ and $s(x)= x^{-1}x$ the range and source maps respectively. We denoted by $\G^{(2)}$ the set of all pairs $(x,y)\in \G\times \G$ such that $s(x)=r(y)$ (that is, all composable pairs). 
We consider maps $f\colon \G\to R$ such that $\im(f)=f(\G)$ is finite.
If $f,g$ are two such maps and $r\in R$, then $f+g, rf\colon \G\to R$ can be defined pointwise: $(f+g)(x)=f(x)+g(x)$ and $(rf)(x)=r f(x)$ for any $x,y\in \G$. Thus the set of such maps is an $R$-module. 

 A {\it topological groupoid} is a groupoid equipped with a topology such that composition and inversion are continuous. The groupoid $\G$ is said to be  {\it étale} if the domain map  $s$ is a local homeomorphism. It is known that when $\G$ is étale, the range map $r$ is also a local homeomorphism. An open subset $B$ of a topological groupoid $\G$ is an {\it open bisection} if $r$ and $s$ restricted to $B$ are homeomorphisms onto an open subset of $\G^{(0)}$. 
 The set of all compact open bisections of a groupoid $\G$ will be denoted $\bis(\G)$.
 We say that a topological groupoid $\G$ is {\it ample} if  there is a basis for its topology consisting of compact open bisections. The  étale groupoids with totally disconnected unit spaces are the ample groupoids, see \cite[Proposition~4.1]{Exel}. An ample groupoid is automatically étale, locally compact and $\G^{(0)}$ is an open subset of $\G$.
 
If $\G$ is a groupoid $\G$, its {\it isotropy subgroupoid} is the set 
$ {\rm Iso}(\G) = \cup_{u\in\G^{(0)}}\G_u^u$
where $\G_u^u$ is the set of elements of $\G$ whose range and
source is $u$. A topological groupoid $\G$ will be called {\it effective} if the interior of $ {\rm Iso}(\G)$ coincides with its unit space $\G^{(0)}$. A subset $U$ of the unit
space $\G^{(0)}$ of $\G$ is said to be {\it invariant} if $s(\l)\in U$ implies $r(\l)\in U$, this is equivalent to $r(s^{-1}(U)) = U = s(r^{-1}(U))$. If $U$ is an invariant subset of $\G^{(0)}$ we define $\G_U := s^{-1}(U)$ which turns out to be a groupoid with unit space $U$. A groupoid $\G$ is said to be {\it strongly effective} if for every nonempty closed invariant subset $V$ of $\G^{(0)}$, the groupoid $\G_V$ is effective.
 
 Suppose $\G$ is an ample Hausdorff groupoid and $R$ is a commutative unital ring. The {\it Steinberg algebra} associated to $\G$ denoted $A_R(\G)$ is the $R$-algebra of all locally constant functions $f\colon \G \to R$ such that \[
 \hbox{Supp} f\coloneqq\{x \in \G \colon f(x)\ne 0\}
 \]is compact. So, $\hbox{Supp} f$ is open and compact for $f\in A_R(\G)$ and  
 \[A_R(\G) = \Span{ \{1_B \colon B \hbox{ is a compact open bisection}\}},
 \]
 where $1_B \colon \G \to R$ is the characteristic function of $B$. Addition and scalar multiplication are defined as above (pointwise).  Multiplication is given by convolution
 $f g\colon \G\to R$ such that 
 \[
fg(x)=\sum_{ab=x}f(a)g(b),
\]
 which is a finite sum because functions have compact support. 
 This convolution formula is such that for compact open bisections $B$ and $D$ we have $1_B1_D = 1_{BD}$. 
The charactaristic functions of compact open subsets of the unit space form a collection of local units for the Steinberg algebra.  See, for example, 
\cite[Lemma~2.6]{ECP}.

In an ample groupoid, the collection of all compact open bisections forms an inverse semigroup with respect to the product an inverse defined lemma below.  Study of this inverse semigroup was made explicit in \cite{Paterson} as well as \cite{Steinberg1}, providing an inverse semigroup view of the  introduction of Steinberg algebras. In particular, we have the following lemma, the details of which are provided in \cite[Lemma 2.1, p. 30]{Simon1}:

\begin{lemma} \label{deS}
    Let $\G$ be an étale groupoid where $\G^{(0)}$ is Hausdorff. 
    If $A, B, C\subset \G$ 
are compact open bisections, then
\begin{enumerate}
\item $A^{-1}= \{a^{-1}\colon a\in A\}$ and $AB = \{ab\colon (a, b)\in(A\times B)\cap \G^{(2)} \}$ are compact
open bisections.
\item If $\G$ is Hausdorff, then $A\cap B$ is a compact open bisection.
\end{enumerate}
\end{lemma}

\begin{remark}
\label{Donsimon}\rm The unit space of an ample Hausdorff groupoid is a locally compact Boolean space. See, for example, \cite[Proposition~3.6]{Steinberg1}. 
Recall by 
that
if $B,C\subset\G^{(0)}$, then $BC=B\cap C$. 
\end{remark}

If $A$ is an algebra (with local units), the socle of $A$ (more precisely the left socle) is the sum of all the minimal left ideals of the algebra. The right socle is defined similarly. We know that left socle and right socle agree (\cite[Theorem 1, Chapter IV, \S 3]{jacobson}) in a semiprime algebra but Steinberg algebras are not necessarily semiprime. So we adopt the convention that the socle of an algebra (denoted $\soc(A)$), stands for the left socle. 

Conditions under which a Steinberg algebra is semiprime have been investigated in \cite[Theorem 4.8]{Steinberg2}:
if $R$ is a commutative ring with unit and $\G$ an effective Hausdorff ample groupoid, then $A_R(\G)$ is semiprime if and only if $R$ is reduced.
As a consequence, if $\G$
is an effective Hausdorff ample groupoid, and $K$ is a field,  then $A_K(\G)$ is a semiprime algebra. In this case the left socle and the right socle of $A_K(\G)$
agree. For the rest of the manuscript, $K$ will always denote a field. 

\section{Minimal left ideals}
\label{sec:leftideals}
 
 There are many examples of Leavitt path algebras that have non-trivial minimal left ideals, as can be seen in \cite{lpasocle}. The same is true for finite dimensional group algebras (see below for the details). These examples are classes of Steinberg algebras such that the underlying groupoid has an open singleton in the unit space with a trivial or finite isotropy group respectively.  Before formalising this observation in Theorem~\ref{Alhaurin}, we state a standard lemma.  
\begin{lemma}\label{queenstown}
    Let $\G$ be a groupoid and $x,y \in \G^{(0)}$. If $y \G x \neq \emptyset$, then for a fixed element $\beta \in y \G x$, we have $y\G x= \beta x \G x$ and $|x \G x | = |y \G x|$. 
\end{lemma}
\begin{proof}
    Since  $y \G x \neq \emptyset$ there exists $\beta \in y \G x$. Consider the map $\phi \colon x\G x \to y \G x$ such that $\gamma \mapsto \phi(\gamma) := \beta \gamma $. This map is bijective since  $\psi \colon y\G x \to x\G x$ such that $\delta \mapsto \psi(\delta):= \beta^{-1}    \delta$ is its inverse. The result follows.
\end{proof}

The structure of the group algebra of a certain isotropy group is fundamental to our results so we delve into some of the details now. If $G$ is a finite group with $n$ elements, the group algebra $KG$ for $K$ a field is $n$-dimensional. This algebra has nontrivial socle, that is, it possesses minimal left (resp. right)  ideals.  In particular, using the Steinberg algebra notation,  the ideal $K(\sum_{g \in G}1_{\{g\}})$ is minimal. Defining $f\coloneqq\sum_{g \in G}1_{\{g\}}$, observe that $f^2=nf$. Other minimal left ideals are present in two flavours:
\begin{itemize}
\item[(i)] Those generated by a division idempotent, that is, an idempotent $a$ such that $a(KG)a$ is a division ring. This case occurs precisely when the characteristic of $K$ does not divide $n$ because then $f'=\frac{1}{n}f$ is a division idempotent.
\item[(ii)] Those generated by an absolute zero divisors, that is, an element $a$ such that $a(KG)a =0$. Concretely, if the characteristic of $K$ divides $n$, then $f$ itself is an absolute zero divisor. In particular, $f$ is nilpotent of nilindex $2$.
\end{itemize}
In the situation where $KG$ is a semiprime algebra, the unique zero absolute divisor is $0$. Hence, minimal left ideals of the second type (ii) do not exist. This dichotomy appears also in our study, as we show in the following theorem.

\begin{theorem}\label{Alhaurin}
    Let $\G$ be an ample Hausdorff groupoid and $x \in \G^{(0)}$. If $\{x\}$ is a compact open set such that $|x\G x|= n < \infty$, then the left ideal $I = A_K(\G) \sum_{\alpha \in x\G x} 1_{\{\alpha\}}$  is minimal. Furthermore:
    \begin{enumerate}
        \item If the characteristic of $K$ does not divide $n$, then $e:=\frac{1}{n}\sum_{\a\in x\G x}1_{\{\a\}}$ is an idempotent.
        \item If the characteristic of $K$  divides $n$, then $f=\sum_{\alpha \in x\G x} 1_{\{\alpha\}}$ is a nilpotent element of nilindex $2$.\end{enumerate}
\end{theorem}
\begin{proof}
    Let $x\G x = \{\alpha_1, \alpha_2,\ldots,\alpha_n\}$. 
     Since $\{x\}$ is a compact open set, each $\{\alpha_i\}$ is  a compact open bisection  and the left ideal 
     \[
     I = A_K(\G) \sum_{\alpha \in x\G x} 1_{\{\alpha\}}=A_K(\G) \sum_{i=1}^n 1_{\{\alpha_i\}}
     \]
     is well defined. In order to prove that $I$ is minimal, we show that for every $f \in A_K(\G)$ such that $f * \sum_{i=1}^n1_{\{\alpha_i\}}\neq 0$, there exists $g \in A_K(\G)$ such that $g*f*\sum_{i=1}^n1_{\{\alpha_i\}} = \sum_{i=1}^n1_{\{\alpha_i\}}$. Take such and $f$ and write is as $f = \sum_{l=1}^t b_l1_{B_l}$ with $B_l$ a compact open bisection and $b_l \in K$ for every $l=1,\ldots,t$. Since 
     \[
     f *\sum_{i=1}^n1_{\{\alpha_i\}} = f *1_{\{x\}}*\left (\sum_{i=1}^n1_{\{\alpha_i\}} \right )\neq 0,
     \]
     then $f*1_{\{x\}} \neq 0$ and without loss of generality we can assume that $f * 1_{\{x\}} = f$, that is, $f = \sum_{l=1}^tb_l1_{\{\delta_l\}}$ such that each $s(\delta_l)=x$.

     After a possible reordering of terms, we can break this sum into groups corresponding to the ranges of the $\delta_l$'s, that is
     \[
     f = \sum_{l = 1}^{t_1}b_l1_{\{\delta_l\}}+\sum_{l=t_1+1}^{t_2}b_l1_{\{\delta_l\}}+\cdots +\sum_{l=t_{m-1}+1}^{t_m} b_l 1_{\{\delta_l\}}
     \]
     with $r(\delta_{l})= x_k$ for every $t_k+1 \leq l \leq t_{k+1} $ and $x_i \ne x_j$ if $i\ne j$. Using Lemma~\ref{queenstown}, we can find $\{\gamma_1,\ldots,\gamma_m\} \subset \G$ such that $r(\gamma_j) = x_j$ and $s(\gamma_j) = x$ for every $j = 1,\ldots,m$ and write 
     \[
     f = \sum_{j = 1}^m\sum_{k=1}^{n} a_{jk}1_{\{\gamma_j\alpha_k\}}
     \]
     where $a_{jk} \in K$ for every $j= 1,\ldots,m$ and $k = 1,\ldots,n$.  If we multiply $f*\sum_{i=1}^n 1_{\{\alpha_i\}}$ we have that 
    \begin{align*}
         f*\sum_{i=1}^n 1_{\{\alpha_i\}} &=\sum_{j=1}^m\sum_{i,k= 1}^{n}a_{jk}1_{\{\gamma_j\alpha_k\alpha_i\}} \\&= \sum_{j=1}^m\sum_{k=1}^{n} a_{jk}1_{\{\gamma_j\}}\sum_{i=1}^n 1_{\{\alpha_k\alpha_i\}}.
    \end{align*}

The sum $\sum_{i=1}^n 1_{\{\alpha_k\alpha_i\}}$ is just a reordering of the original sum $\sum_{i=1}^n 1_{\{\alpha_i\}}$ and hence 
    \begin{equation*}
        f*\sum_{i=1}^n 1_{\{\alpha_i\}} = \left(\sum_{j=1}^{m}\left (\sum_{k=1}^n a_{jk} \right )1_{\{\gamma_j\}}\right ) \left (\sum_{i=1}^n1_{\{\alpha_i\}}\right ). 
    \end{equation*}
    Since $f * \sum_{i=1}^n1_{\{\alpha_i\}} \neq 0$ there is at least some $j$ such that $\sum_{k=1}^n a_{jk} \neq 0$. Without loss of generality we can fix $j = 1$ and take $g = \frac{1}{\Omega} 1_{\{\gamma_1^{-1}\}}$ with $\Omega = \sum_{k=1}^n a_{1k} \neq 0$. Then we have that
    \begin{equation*}
        g * f * \sum_{i=1}^n 1_{\{\alpha_i\}} = \frac{1}{\Omega} \Omega 1_{\{x\}}\left (\sum_{i=1}^n1_{\{\alpha_i\}}\right ) = \sum_{i=1}^n1_{\{\alpha_i\}}. 
    \end{equation*}
    If $\text{char}(K)$ does not divide $n$, then  the element $e:=\frac{1}{n}\sum_{\a\in x\G x}1_{\{\a\}}$  is an idempotent since
    \[e^2=\frac{1}{n^2}\sum_{i=1}^n\sum_{j=1}^n 1_{\{\a_i\a_j\}}=\frac{1}{n^2}n\sum_{i=1}^n1_{\{\alpha_i\}}= e,\]
    otherwise $\left(\sum_{i = 1}^n1_{\{\a_i\}}\right)^2=0$. 
\end{proof}

\begin{remark}\rm
    As a consequence of Theorem~\ref{Alhaurin}, a sufficient condition for the existence of a
nonzero socle in a
Steinberg algebra (of an ample Hausdorff groupoid), is the existence of an $x\in \G^{(0)}$ such that  $\{x\}$ is a compact open set with
$x\G x$ finite. 
\end{remark}

\begin{remark}\label{tejeringo}\rm
    Theorem \ref{Alhaurin} says that the minimal left ideal $I$ is presented in two flavors: in the first one, it is generated by a division idempotent, while in the second one, the generator is an absolute zero divisor. 
If the characteristic of the ground field does not divide $n$, the element $e=\frac{1}{n}\sum_{i=1}^n 1_{\{\alpha_i\}}$ is an idempotent of $A_K(\G)$. So the left ideal generated by $e$, that is, $I=A_K(\G)e$, which is a minimal left $A_K(\G)$-module, is such that $\hbox{End}_{A_K(\G)}(I)$ is a division $K$-algebra.
But $\hbox{End}_{A_K(\G)}(I)\cong e A_K(\G) e$ as $K$-algebras. So the corner
$e A_K(\G) e$ is also a division $K$-algebra. That is why we say that $e$ is a division idempotent.

On the other hand, if the characteristic of $K$ divides $n$, then Theorem~\ref{Alhaurin} says $\sum_{i=1}^n 1_{\{\a_i\}}$ is nilpotent of nilindex $2$.
Furthermore, in this case, the left ideal $I$ satisfies $I^2=0$. This is because $\sum_{i=1}^n1_{\{\a_i\}}$ is an absolute divisor of zero, that is, $\sum_{i =1}^n1_{\{\a_i\}}A_K(\G) \sum_{j=1}^n1_{\{\a_j\}}=0$. Equivalently, we must prove that
$\sum_{i=1}^n1_{\{\a_i\}} 1_B \sum_{j=1}^n1_{\{\a_j\}}=0$ for any compact open bisection $B$. Thus we
  need to check that for any  $\g\in x\G x$ we have 
$\sum_{i,j=1}^n 1_{\{\a_i\g\a_j\}}=0$.
 But
 we have $\sum_{i,j=1}^n 1_{\{\a_i\g\a_j\}}=\sum_{i=1}^n 1_{\{\a_i\}}\sum_{j=1}^n1_{\{\g\a_j\}}=\sum_{i=1}^n 1_{\{\a_i\}}\sum_{j=1}^n 1_{\{\a_j\}}=0$.
\end{remark}

In the case where the isotropy at a unit $x$ in Theorem~\ref{Alhaurin} has trivial isotropy, the following corollary is immediate.  

\begin{corollary}\label{cor_1xminimal}
    Let $\G$ be an ample Hausdorff groupoid and let $x \in \G^{(0)}$ such that $\{x\}$ is open and $x \G x = \{x\}$. Then $A_K(\G)1_{\{x\}}$ is a minimal left ideal and $1_{\{x\}}A_K(\G)1_{\{x\}} = K 1_{\{x\}}$.
\end{corollary}

In an associative algebra with local units, like a Steinberg algebra, all minimal left (resp. right) ideals are of the form $A_K(\G)a$ (resp. $aA_K(\G)$) for some $a \in A_K(\G)$. When first approaching the problem,  in order for the ideal to be minimal, we wanted the element $a$ has to be ``very small''. Our intuition says that if the compact open  bisections that define $a$ are very small (very few elements in them), then $a$ will also be small and possibly the ideal that it generates it will be minimal. However, there is the possibility that the ideal is not generated by a ``small'' element but it is isomorphic to one. In this sense, we describe a result analogous to \cite[Theorem 4.13]{socletheory} (see also \cite[Lemma 3.1]{kpsoc}). First we stablish two lemmas.

\begin{lemma}\label{lem:new} Let $\G$ be an ample Hausdorff groupoid and let $A_K(\G)b$ be a minimal left ideal for some $b \in A_K(\G)$.  Then there exists $a \in A_K(\G)$ with $A_K(\G) b \cong A_K(\G)a$ as left $A_K(\G)$-modules such that 
 $a$ has the form 
 \[a = a_01_C+\sum_{i=1}^na_i1_{B_i}\]
 with $C \subset \G^{(0)}$, $B_i \subset \G\setminus \G^{(0)}$, $B_i\cap B_j = \emptyset$ for every $i \neq j$ and $r(B_i) \subseteq C$.
\end{lemma}
\begin{proof}
     Write $b = \sum_{i=0}^n a_i1_{B_i}$ with $a_i \in K$,  $B_i$ non-empty compact open  bisections and 
 $B_i\cap B_j = \emptyset$ for every $i \neq j$. Then we have 
 \[a \coloneqq 1_{B_0^{-1}}b = a_01_{s(B_0)}+\sum_{i=1}^n a_i 1_{B_0^{-1}B_i}.\]
 Since $A_K(\G)a \subseteq A_K(\G)b$ and the latter is minimal, the result follows.  
\end{proof}

\begin{lemma}\label{lemma_sandwich}
    Let $\G$ be an ample Hausdorff groupoid and let $A_K(\G)a$ be a minimal left ideal of $A_K(\G)$ for some $a\in A_K(\G)$. Then there is a compact open set $C \subset \G^{(0)}$ and an element $b \in 1_C A_K(\G)1_C$ such that $A_K(\G)a \cong A_K(\G)b$ as left $A_K(\G)$-modules. 
\end{lemma}
\begin{proof}
By Lemma~\ref{lem:new}, we can assume that $a$ has the form 
 \[a = a_01_C+\sum_{i=1}^na_i1_{B_i}\]
 with $C \subset \G^{(0)}$, $B_i \subset \G\setminus \G^{(0)}$, $B_i\cap B_j = \emptyset$ for every $i \neq j$ and $r(B_i) \subseteq C$.
Thus $1_Ca = a$ and  
\[1_Ca1_C= a_01_C + \sum_{i=1}^na_1 1_{B_iC} \neq 0.
\]
If we take $b = 1_Ca1_C \in 1_CA_K(\G)1_C$, we have that the map \[
R_{1_C}\colon A_K(\G)1_Ca \to A_K(\G)b 
\quad
\text{such that} \quad R_{1_C}(f) = f1_C
\] is a module isomorphism: it is injective by the minimality of $A_K(\G)1_Ca$ and surjective by construction. 
\end{proof}

Under the conditions of the previous lemma, we next show that minimal left (resp. right) ideals ensure the existence of open singletons in $\G^{(0)}$.

\begin{proposition}\label{prop_existence_singleton}
 Let  $\G$ be an ample Hausdorff groupoid and let $I = A_K(\G)a$ be a minimal left ideal for some $a \in A_K(\G)$ as in Lemma~\ref{lem:new}. Then there is an element $u \in \G^{(0)}$ such that $\{u\}$ is open and  $u\in\Supp(a)$.
\end{proposition}
\begin{proof}
Since $I$ is a minimal left ideal we have that $I =A_K(\G)a$ with $a\in I$ and $a = a_01_C+ \sum_{i=1}^la_i 1_{B_i}$ with $C \subset \G^{(0)}$, $B_i \subset \G \setminus \G^{(0)}$ and $a_0 \neq 0$ as in Lemma~\ref{lem:new}. Since $C$ is nonempty, there is a unit $u \in C \subset \supp(a)$. We show that $\{u\}$ is open. 

By way of contradiction, suppose that $\{u\}$ is not open, then $\G^{(0)} \setminus \{u\}$ is not closed. Since $\G$ is Hausdorff,  $\G^{(0)}$ is closed in $\G$ and so there must be a sequence $(x_n)_{n \geq 1} \subset \G^{(0)}\setminus \{u\}$ such that $x_n \to u$ and $x_i \neq x_j$ for every $i \neq j$.  Since $u \in \supp(a) \cap \G^{(0)}$, which is open, and $x_n \to u$ we can assume that $(x_n)_{n\geq 1} \subset \G^{(0)}\cap \supp(a) \setminus \{u\}$. Consider $x_1 \neq u$.  Since $\G$ is Hausdorff there exist compact open sets $V_1$ and $U_1$ in $\G^{(0)}$ such that $x_1 \in V_1 \subset \G^{(0)}\cap \supp(a)$, $u \in U_1\subset \G^{(0)}\cap \supp(a) $ and $V_1 \cap U_1 = \emptyset$. Then, we have that $1_{V_1}a \neq 0$ and $1_{U_1}a \neq 0$. Since the ideal $I$ is minimal, this implies that $I = A_K(\G)1_{V_1}a = A_K(\G)1_{U_1}a$. In particular, there is a function $g \in A_K(\G)$ such that $g1_{V_1}a = 1_{U_1}a$. We evaluate both sides at $u$.  For the left-hand side,  we have that $1_{U_1}a(u) = a(u) \neq 0$.  For the right-hand side, we have 
\[
g1_{V_1}a(u) = \sum_{u = \gamma \beta \delta} g(\gamma)1_{V_1}(\beta) a(\delta) \neq 0.\]
Since the sum is nonzero, there is at least one $\gamma,\beta, \delta \in G$ with $u = \gamma \beta \delta$ and $g(\gamma)1_{V_1}(\beta)a(\delta) \neq 0$. Thus $\beta \in V_1 \subset \G\0 \setminus U_1$ and we have that $\alpha_1:= \delta$ and $\alpha^{-1}:= \gamma$. In summary, we have found an element $\alpha_1 \in \G$ such that $s(\alpha_1) = u$, $r(\alpha_1) \in V_1$ and $g(\alpha_1^{-1})a(\alpha_1) \neq 0$, that is, $\alpha_1 \in \supp(a)$. In addition, there is an element $n_1$ such that $x_n \in U_1$ for every $n \geq n_1$. 

By repeating this process,  we find compact open sets  $V_2,U_2 \subset \G^{(0)}$ such that $x_{n_1} \in V_2$, $u \in U_2$, $V_2 \cap U_2 = \emptyset$ and $U_2 \cup V_2 \subset U_1$. Again, by the minimality of the ideal $I$ this implies that there is some $\alpha_2 \in \G$ such that $s(\alpha_2) = u$ and $r(\alpha_2) \in V_2$ and $a(\alpha_2) \neq 0$. If we iterate this process infinitely many times, we can construct a sequence $(\alpha_n)_{n\geq 1} \subseteq \supp(a)$ with $\alpha_n \neq \alpha_m$ for $n \neq m$, $s(\alpha_n) = u$ and $r(\alpha_n) \in V_n$ for every $n \in V_n$. However, $a = a_01_C + \sum_{i= 1}^la_i1_{B_i}$ and its support contains at most $l+1$ different elements with source $u$, which is a contradiction.
\end{proof}

\begin{corollary}\label{cor_minimalsingleton}
If $\G$ is an ample Hausdorff groupoid and $I$ is a minimal left ideal of $A_K(\G)$, then there is a compact open singleton $\{x\} \subset \G^{(0)}$ such that $I \cong A_K(\G)a$ as left $A_K(\G)$-modules for some $a \in 1_{\{x\}}A_K(\G)1_{\{x\}}$.
\end{corollary}
\begin{proof}
    Consider $I = A_K(\G)b$ a minimal left ideal where $b = \sum_{i=1}^n r_i 1_{B_i} \neq 0$ with $r_i \in K$, $B_i$ compact open bisection such that $B_i \cap B_j = \emptyset$. By Proposition \ref{prop_existence_singleton} there is a unit $x \in \supp(b)$ such that $\{x\}$ is open. Thus, $a\coloneqq 1_{\{x\}}b1_{\{x\}} \neq 0$ since $a = \sum_{i=1}^n r_i 1_{xB_ix}$ with $xB_ix \cap x B_j x = \emptyset$. Since $I$ is minimal, we have 
    \[
    I= A_K(\G)b = A_K(\G)1_{\{x\}}b \cong A_K(\G)1_{\{x\}}b1_{\{x\}} =A_K(\G)a.
    \]
\end{proof}
So far we have shown that, on the one hand, \ref{prop_existence_singleton} implies that the existence of minimal left ideals is related to the existence of open  singletons. On the other hand, Theorem \ref{Alhaurin} implies that under certain conditions of the isotropy group of such singletons, we can form minimal left ideals. However, not every open singleton implies the existence of a minimal left ideal. As an example, we will see that if $\{x\}\subset \G\0$ is a open and such that $x\G x \cong \ZZ^k$, then $A_K(\G)1_{\{x\}}$ is not a minimal left ideal. This result and its corresponding proof, is a translation of  \cite[Proposition 2.5]{pino2008socle} into the Steinberg algebra context. This is a consquence of Corollary~\ref{cor:Zk} below but we first establish some preliminary results.
\begin{lemma}\label{nomaslabels}
    If $\G$ is an ample Hausdorff groupoid and $\{x\}$ is open in $\G^{(0)}$ with $x\G x \cong \mathbb{Z}^k$, then $1_{\{x\}}A_K(\G)1_{\{x\}} \cong K[y_1, \dots, y_k,y^{-1}_1, \dots, y^{-1}_k]$ as $K$-algebras.
\end{lemma}
\begin{proof}
    If $x\G x\cong\Z^k$, then there are elements $\g_i\in x\G x$ ($i=1,\ldots,k)$ such that any $a\in x\G x$ is of the form $a=\g_1^{e_1}\cdots \g_k^{e_k}$ for some integers $e_i$. It is then enough to consider the $K$-algebra isomorphism $\phi \colon 1_{\{x\}}A_K(\G)1_{\{x\}} \to  K[y_1,\ldots,y_k,y_1^{-1},\ldots,y_k^{-1}]$ with $\phi\left (\prod_i1_{\{\g_i^{e_i}\}} \right ) = \prod_i y_i^{e_i}$.
\end{proof}

Before we state out main theorem, which puts together the results we have established so far, we spend a little time considering some general facts about minimal ideals. 

\begin{proposition}\label{general} Let $R$ be a ring with local units, $e\in R$ a nonzero idempotent and $a\in eRe$. Then if $Ra$ is minimal left ideal of $R$, the left ideal $eRa$ of $eRe$ is also minimal. \end{proposition}

\begin{proof} We have $eRa\ne 0$ since on the contrary $a=eea\in eRa=0$ implying $Ra=0$, which contradicts minimality. 
Next we prove that $eRa$ is minimal in $eRe$.
Let $0\ne J\subset eRa$ be a left ideal of $eRe$. Notice that $J=eJ$. 
Since $R$ has local units, $RJ\ne 0$ and $RJ\subset Ra$ and hence by minimality $RJ=Ra$. Thus
$a\in RJ$ again using that $R$ has local units. So $a=ea\in eRJ=eReJ\subset J$ and consequently 
$eRa=eRea\subset eRe J\subset J$. So $J=eRa$ proving the minimality of $eRa$.
\end{proof}

Now we take advantage that in general, the socle of a field is the field itself and the socle of a domain that is
not a field is $0$.

\begin{corollary}
Under the hypothesis of the previous proposition, if $eRe$ is a domain but not a field, $Ra$ is not a 
minimal left ideal.
\end{corollary}
\begin{proof} If $Ra$ were minimal then $eRa$ would be a minimal left ideal of $eRe$ hence $\soc(eRe)\ne 0$,
a contradiction.  
\end{proof}

From these facts, we deduce the following:
 Let $\G$ be an ample Hausdorff groupoid. If  $\{x\}$ is open in $\G^{(0)}$ and $x\G x \cong \mathbb{Z}^k$, then there are no elements $a \in 1_{\{x\}}A_K(\G)1_{\{x\}}$ 
that generate a minimal left ideal.
Taking into account that  we have shown that $1_{\{x\}}A_K(\G)1_{\{x\}}$ is a domain but not a field in Lemma \ref{nomaslabels}, we obtain the next corollary. 
\begin{corollary}\label{cor:Zk}
     Let $\G$ be an ample Hausdorff groupoid and let $\{x\}$ be open in $\G^{(0)}$ with $x\G x \cong \mathbb{Z}^k$.  Then there are no  elements $a \in 1_{\{x\}}A_K(\G)1_{\{x\}}$ that generate a minimal left ideal. 
\end{corollary}

\begin{example} \label{ex:kp}We apply our results to Kumjian-Pask algebras and generalise \cite[Theorem~3.9]{kpsoc} from row-finite $k$-graphs to finitely aligned $k$-graphs. See \cite{CP} for the details of the Steinberg algebra approach for these algebras.   Given a finitely aligned $k$-graph $\Lambda$, write $\G_{\Lambda}$ for the boundary path groupoid of $\Lambda$ and identify its unit space with $\delta \Lambda$.     We extend the definition of vertices that are \emph{line points} as follows:
\[
P_l(\Lambda)\coloneqq \{ v \in \Lambda^{0} : v\delta \Lambda = \{x\} \text{ and } x \text{ is aperiodic}\}
\]  Recall that $x \in \delta \Lambda$ is aperiodic if and only if the isotropy at $x$ is trivial; otherwise, it is a subgroup of $\mathbb{Z}^k$. 
So when $v \in P_l(\Lambda)$, there exists a unit $x$ such that 
$\{x\}=v\delta\Lambda=Z(v)$ is clopen.  
 \begin{corollary}\label{thm:kp}
 Let $\Lambda$ be a finitely aligned $k$-graph and let $K$ be a field. 
 \begin{enumerate}
     \item\label{it1:kp} Let $v \in \Lambda^0$. Then KP$_K(\Lambda)p_v$ is a minimal left ideal if and only if $v \in P_l(\Lambda)$.
     \item\label{it2:kp} Let $a \in$ KP$_K(\Lambda)$. Then KP$_K(\Lambda)a$ is a minimal left ideal if and only if there exists $v \in P_l(\Lambda)$ such that KP$_K(\Lambda)p_v$ is isomorphic to KP$_K(\Lambda)a$ as left KP$_K(\Lambda)$- modules.
 \end{enumerate}
 \end{corollary}
 \begin{proof}
     For item~\eqref{it1:kp}, the forward implication is a consequence of Proposition~\ref{prop_existence_singleton} and Corollary~\ref{cor:Zk}.  The reverse implication is given by Corollary~\ref{cor_1xminimal}.  For item~\eqref{it2:kp}, the forward implication is given by Corollary~\ref{cor_minimalsingleton} and a computation.  The reverse implication follows from item~\eqref{it1:kp}.
 \end{proof}
 \end{example}

To summarize,  we know that minimal left ideals imply the existence of open singletons.  Indeed, if we have a minimal left ideal, it will be isomorphic to one that arises from an open singleton in the unit space.  On the other hand, if  we impose some conditions on the isotropy groups, we know that open singletons give rise to minimal ideals. However, one of the questions that remains if this is necessary and sufficient. The answer is no.  We need additional conditions on the groupoid to be able to characterise all minimal ideals. As we pointed out earlier, Steinberg algebras generalise group algebras, and characterising the socle of these algebras, in general, is tough. 

\begin{remark}\label{remark_local_socle}{\rm
    Suppose we have an ample Hausdorff groupoid  $\G$ with open $\{x\}\subset \G\0$ and $I = A_K(\G)a$ for some $a \in 1_{\{x\}}A_K(\G)1_{\{x\}}$ where $I$ is a minimal left ideal. Take \[
    J=1_{\{x\}}I1_{\{x\}} = 1_{\{x\}}A_K(\G)1_{\{x\}}a.
    \]
    Then $J$ is nonzero and we know that $J$ is a minimal left ideal of the algebra $1_{\{x\}}A_K(\G)1_{\{x\}}$ by Proposition \ref{general}.
Therefore, if the algebra $1_{\{x\}}A_K(\G)1_{\{x\}}$ does not contain minimal left ideals, then $I$ can not be minimal. In this way, we have a necessary condition for the existence of minimal left ideals, that is, if $\soc(A_K(\G)) \neq 0$,  then there exists open $\{x\} \subset \G\0$  such that $\soc(1_{\{x\}}A_K(\G)1_{\{x\}})\neq 0$. Next, we ask ourselves is if there is any characterisation of the minimal left ideals of the algebra $1_{\{x\}}A_K(\G)1_{\{x\}}$. The answer of that depends on the structure of the isotropy group $x \G x$.}
\end{remark}
 \begin{remark}\label{lemma_corner_steinberg}{\rm 
     Taking into account \cite[Proposition 2.3]{Abramscorner}, we have that if $\G$ is an ample Hausdroff groupoid and $x \in \G\0$ with $\{x\}$ open, then the algebra $1_{\{x\}}A_K(\G)1_{\{x\}}$ is isomorphic to $A_K(x\G x)$, that is, the Steinberg algebra of the isotropy group $x\G x$. }
 \end{remark}

Remark \ref{remark_local_socle} and Remark \ref{lemma_corner_steinberg} imply that the minimal ideals of a Steinberg algebra are related to the minimal ideals of a group algebra. In general, finding the minimal ideals of a group algebra is an open question. There are partial answers, but no complete characterisation. The reader can find some of them in \cite{BENSON2023459,farkas1981group,richardson1976group}. In this sense, and with the objective of generalising  the socle of the Leavitt path and Kumjian-Pask algebras, from here, we focus on the case that $\G$ is an ample Hausdorff groupoid such that for every $x \in \G\0$, one has: $\soc(A_K(x\G x)) \neq \{0\}$ if and only if $x\G x = \{x\}$.

\section{The socle of a Steinberg algebra}
\label{sec:socle}

Taking into account all the results obtained in the previous paragraphs, we have established the previous section, we have the following theorem.

\begin{theorem}\label{puerto} Let $\G$ be an ample Hausdorff groupoid satisfying the condition:
\begin{equation}\label{conditionpotente}
    \text{for every
     } x \in \G^{(0)}, \; \soc(A_K(x\G x))\neq 0 \Leftrightarrow x\G x = \{x\}. \tag{LP}
\end{equation}
Then the following statements are equivalent: 
\begin{enumerate}
    \item\label{it1:main} $\soc(A_K(\G))\neq 0$.
    \item\label{it2:main} There exists an open set $\{x\} \subset \G\0$ such that $x \G x=\{x\}$.
    \item\label{item3:main} There exists an open set $\{x\} \subset \G\0$ such that $x \G x$ is finite.
\end{enumerate}
\end{theorem}

\begin{proof}
Corollary~\ref{cor_minimalsingleton} gives item~\eqref{it1:main} implies
item~\eqref{it2:main}.  Item~\eqref{it2:main} implies item~\eqref{item3:main} is trivial and item~\eqref{item3:main} implies item~\eqref{it1:main} is Theorem~\ref{Alhaurin}.
\end{proof}

Observe that, in general there are many examples of groupoids that do not satisfy Condition \eqref{conditionpotente} even though Corollary \ref{cor_1xminimal} leads to one of the implications: if $x\G x = \{x\}$, then $\soc(A_K(x\G x)) = K1_{\{x\}} \neq 0$. Indeed, Theorem \ref{Alhaurin} it is enough to consider a groupoid with a discrete topology such that there is some $x \in \G^{(0)}$ with $|x\G x | > 1$.

The groupoids associated to Kumjian-Pask algebras, and hence Leavitt path algebras, satisfy Condition~\eqref{conditionpotente} of Theorem \ref{puerto}. Indeed, as we saw in Example~\ref{ex:kp}, if $x \in \G_{\Lambda}\0$ for a $k$-graph $\Lambda$ (including the Leavitt path algebra setting where $k=1$), the isotropy group $x\G_{\Lambda} x = \{x\}$ or $x \G x \cong \ZZ^l, l\leq k$. However, when $x\G x \cong \ZZ^l$, we have that $\soc(K\ZZ^l) = \{0\}$ and we are under the claimed condition.

Next we focus on the description of the socle of general Steinberg algebras satisfying Condition~\eqref{conditionpotente}.
In the remainder of this section, we give our final characterisation of the socle. In the next theorem we will use the notation $(S)$ to denote the two-sided ideal of an algebra generated by the set $S$.

\begin{theorem}
\label{thm:main}
    Let $\G$ be an ample Hausdorff groupoid satisfying Condition~\eqref{conditionpotente}. Then the left socle of the algebra $A_K(\G)$ is the ideal 
    \begin{equation}
        \soc(A_K(\G)) = \left (1_{\{x\}} \colon \{x\} \subset \G\0 \text{ is open and }  x\G x = \{x\}\right ).
    \end{equation}
    If $I_{x} = A_K(\G)1_{\{x\}}$ is a minimal left ideal, then its corresponding homogeneous component is
\begin{equation}
    M = \sum_{I \cong I_x} I = (1_{\{x\}}).
\end{equation}
\end{theorem}
\begin{proof}
    Let $I$ be a minimal left ideal. We can consider that $I = A_K(\G)a$ and by Proposition \ref{prop_existence_singleton} there is a compact open bisection $\{x\}$ with $x \in \supp(a)$. This means that $1_{\{x\}}a \neq 0$ and by minimality, we have that $I = A_K(\G)1_{\{x\}}a$. Not only that but there is a $A_K(\G)$-modules isomorphism $I \cong A_K(\G)1_{\{x\}}a1_{\{x\}} =: J$. Since $J$ is a minimal left $A_K(\G)$-module, then $1_{\{x\}}J$ it is also a minimal left $1_{\{x\}}A_K(\G)1_{\{x\}}$-module. In consequence, by Remark \ref{lemma_corner_steinberg}, $\soc(A_K(x\G x)) =  \soc(1_{\{x\}}A_K(\G)1_{\{x\}}) \neq \{0\}$ which is only possible if and only if $x \G x = \{x\}$. To sum up, we have that $I = A_K(\G)1_{\{x\}}a$ with $\{x\} \subset \G\0$ and $x\G x = \{x\}$, then $I \subseteq \left (1_{\{x\}} \colon \{x\} \subset \G\0, x\G x = \{x\} \right )$ and 
    $$\soc(A_K(\G)) \subseteq \left (1_{\{x\}} \colon \{x\}\subset \G\0, x\G x = \{x\}\right ).$$ 
\color{black}
   Now, consider $x \in \G^{(0)}$ with $\{x\}\subset \G^{(0)}$ and $x \G x = \{x\}$. Then, by Corollary \ref{cor_1xminimal} we have that $I_x = A_K(\G)1_{\{x\}}$ is a minimal left ideal and $I_x \subset \soc(A_K(\G))$. Thus $1_\{x\}\in\soc(A_K(\G)$ and since the socle is an ideal we get 
   $$\left (1_{\{x\}} \colon \{x\} \subset \G\0, x\G x = \{x\}\right ) \subseteq \soc(A_K(\G)).$$

    In order to prove the second part take an ideal $I$ such that $I \cong I_x$. As we have mentioned before, we can consider $I = A_K(\G)a$ with $a \in I$. If $\phi \colon I \to I_x$ is a $A_K(\G)$-module isomorphism, then we have that $\phi(a) = f1_{\{x\}}$ and $\phi^{-1}(1_{\{x\}}) = g a$, for some $f,g\in A_K(\G)$. 
 Making use of this fact we get $a = \phi^{-1}(\phi(a))=\phi^{-1}( f1_{\{x\}}) = \phi^{-1}( f1_{\{x\}}1_{\{x\}}) =f1_{\{x\}}\phi^{-1}(1_{\{x\}}) = f1_{\{x\}}ga \in (1_{\{x\}})$ and $I \subset (1_{\{x\}})$. As a consequence, $M \subset (1_{\{x\}})$.

    For the converse, since $I_x$ is a minimal left ideal we have that $I_x \subset M$. And because $M$ is a two-sided ideal we get that $(1_{\{x\}}) = A_K(\G)1_{\{x\}}A_K(\G) = I_x A_K(\G) \subset M$.
\end{proof}

In the description of the homogeneous components, there are elements that generate the ideal that are redundant. We can give a more precise result as follows.

\begin{remark}{\rm 
    It is important to highlight that we can be more precise with the description of the socle. If we define the equivalence relation $\mathcal{R}$ as $x \mathcal{R} y$ if and only if there is $\gamma \in \G$ such that $r(\gamma) = x$ and $s(\gamma) = y$. In this case, we have that if $\{x\} \subset \G^{(0)}$ and $x \mathcal{R} y$ then $(1_{\{x\}}) = (1_{\{y\}})$. If $x \mathcal{R} y$, then there is $\gamma$ with $r(\gamma) = x$ and $s(\gamma)= y$.  The rest is immediate from the fact that $1_{\{x\}} = 1_{\{\gamma\}}1_{\{y\}}1_{\{\gamma^{-1}\}}$ and $1_{\{y\}} = 1_{\{\gamma^{-1}\}}1_{\{x\}}1_{\{\gamma\}}$. As a consequence, we can describe the socle as
    \begin{equation}
        \soc(A_K(\G)) = (1_{\{x\}} \colon [x] \in \G^{(0)}/\mathcal{R}, x \G x = \{x\} \text{ is open}). 
    \end{equation}}
\end{remark}

\begin{proposition}\label{componentsocle}
    Let $\G$ be a  Hausdorff ample groupoid  satisfying Condition ~\eqref{conditionpotente}. If $M$ is the homogeneous component of the socle corresponding to $[x]$, then 
    \begin{equation*}
        M = (1_{\{x\}}) = \bigoplus_{y \in [x]} A_K(\G)1_{\{y\}}.
    \end{equation*}
\end{proposition}
\begin{proof}
    First, we check that it is a direct sum: $(A_K(\G)1_{\{y\}}) \cap (\sum_{y_i \neq y} A_K(\G)1_{\{y_i\}}) = \{0\}$ for every $y,y_i \in [x]$ with $y\neq y_i$ for each $i$. If there is an element in the intersection, this means that there are $f,g_i \in A_K(\G)$ such that $f 1_{\{y\}} = \sum_{y \neq y_i}g_i 1_{\{y_i\}}$. Multiplying both sides to the right by $1_{\{y\}}$ we get that $f 1_{\{y\}} = \sum_{y \neq y_i} g_i 1_{\{y_i\}} 1_{\{y\}}= 0$ as we claimed. Define $J =\bigoplus_{y \in [x]} A_K(\G)1_{\{y\}}$. Now, we will prove  $M =  J$. Consider $f 1_{\{y\}} \in J $ and $\gamma \in y\G x$ ($x\G y \neq \emptyset$ because $y$ belongs to the orbit of $x$), we have that $f 1_{\{y\}} = f 1_{\{\gamma\}}1_{\{x\}}1_{\{\gamma^{-1}\}} \in M$. As a consequence, $J \subseteq M$.

    Take $0 \neq f1_{\{x\}}g \in M$, in particular $1_{\{x\}}g \neq 0$ and without loss of generality we can assume $1_{\{x\}}g = g = \sum_{i=1}^n a_i1_{\{\delta_i\}}$ with $r(\delta_i)= x$ and $s(\delta_i) = y_i \in [x]$. Thus,  $f1_{\{x\}}g = \sum_{i=1}^n a_i f1_{\{\delta_i\}} 1_{\{y_i\}} \in J$. Then $M \subseteq J$.
\end{proof}

\begin{remark}\rm 
Following \cite[Proposition 2.8]{Simon1} the map $(\sum_i r_i 1_{B_i})^*:=\sum_i r_i 1_{B_i^{-1}}$ is an involution on $A_K(\G)$, hence taking into account Proposition \ref{componentsocle} the homogeneous component $ M = (1_{\{x\}})$ satisfies $M^*=M$, so the left socle is self-adjoint. Since the involution applied to the left socle is the right socle, then we have the coincidence of both socles.
\end{remark}

Assume that $\G$ is an ample Hausdorff groupoid with $\{x\}\subset \G^{(0)}$ such that $x\G x=\{x\}$. Then we know that
$\text{End}(A_K(\G)1_{\{x\}})\cong 1_{\{x\}} A_K(\G) 1_{\{x\}}$ by 
Remark \ref{tejeringo}. On the other hand, since $x\G x = \{x\}$ we can write
$1_{\{x\}} A_K(\G) 1_{\{x\}}=K 1_{\{x\}}$. We conclude that 
$$\text{End}(A_K(\G)1_{\{x\}})\cong K 1_{\{x\}}.$$

By using \cite[Theorem 1, Chapter IV, \S 3]{jacobson} the homogeneous components of the socle of a semiprime algebra $A$ (with minimal left ideals), are simple algebras with nonzero socle. 
In fact, the minimal left ideals are of the form $I=Ae$ for a division idempotent. Then $\Delta:=eAe$ is a division algebra $e$. If $A$ is simple then $A\cong\text{End}_\Delta(I)$. In the particular case of a homogeneous component $M=(1_{\{x\}})$ of a semiprime Steinberg algebra with nonzero socle we have seen that $\Delta\cong K$. So $M\cong \text{End}_K(I)$ is of the form $\mathcal{M}_n(K)$ where $n$ is finite or infinite. We know that
$n=\dim_K(I)$.

\begin{corollary}
    Let $\G$ be an ample Hausdorff groupoid satisfying Condition ~\eqref{conditionpotente} and such that $A_K(\G)$ is semiprime. Then 
    \begin{equation}
        \soc(A_K(\G)) \cong \bigoplus_{[x] \in \G^{(0)}/\mathcal{R} \colon x \G x = \{x\}} \mathcal{M}_{\left |[x] \right |}(K).
    \end{equation}
\end{corollary}
\begin{proof}
By \cite[Theorem 1, Chapter IV, \S 3]{jacobson} the homogeneous components of the socle are simple algebras. The socle is a completely reducible module in the terminology of Jacobson, hence \cite[Corollary, Chapter IV, \S 2]{jacobson} gives that the socle is the direct sum of its homogeneous components which are of the form $\text{End}_K(A_K(\G)1_{\{x\}})\cong \mathcal{M}_n(K)$ where $n=\dim(A_K(\G)1_{\{x\}})$ (possibly infinite).
In order to compute $n$, note that the set $C=\{1_{\{\g\}}\colon \g\in s^{-1}(x)\}$ is a system of generators of $A_K(\G)1_{\{x\}}$. Indeed, for any compact open bisection $B$ we have $1_B1_{\{x\}}=1_{B\{x\}}=1_{\{\g\}}$ for some $\g\in s^{-1}(x)$. So let us prove that $C$ is a linearly independent set.
Assume $\sum_i r_i 1_{\{\g_i\}}=0$ where the $\g_i\in s^{-1}(x)$. Take one $\g_k$ of those in the previous sum and compute $1_{\{\g_k^{-1}\}}\sum_i r_i 1_{\{\g_i\}}=0$.
We have $0=\sum_i r_i 1_{\{\g_k^{-1}\}}1_{\{\g_i\}}=\sum_i r_i 1_{\{\g_k^{-1}\g_i\}}=r_k 1_{\{x\}}$ implying $r_k=0$. Thus $C$ is a basis of the ideal $A_K(\G)1_{\{x\}}$. Then the homogeneous component $\text{End}_K(A_K(\G)1_{\{x\}})\cong \mathcal{M}_n(K)$.
\end{proof}

{\bf Conflict of interest statement:}
All authors declare that they have no conflicts of interest. 

{\bf Data availability statement:}  The authors confirm that the data supporting the findings of this study are available within the article.

\bibliographystyle{plain} 
\bibliography{ref}

\begin{thebibliography}{10}

\bibitem{Abramscorner}
G.~Abrams, M.~Dokuchaev, and T.~G. Nam.
\newblock Realizing corners of {L}eavitt path algebras as {S}teinberg algebras,
  with corresponding connections to graph {$C^\ast$}-algebras.
\newblock {\em J. Algebra}, 593:72--104, 2022.

\bibitem{socletheory}
G.~Aranda~Pino, D.~Mart\'in~Barquero, C.~Mart\'in~Gonz\'alez, and
  M.~Siles~Molina.
\newblock Socle theory for {L}eavitt path algebras of arbitrary graphs.
\newblock {\em Rev. Mat. Iberoam.}, 26(2):611--638, 2010.

\bibitem{pino2008socle}
G.~Aranda~Pino, D.~Martín~Barquero, C.~Martín~Gonz{\'a}lez, and
  M.~Siles~Molina.
\newblock The socle of a {L}eavitt path algebra.
\newblock {\em Journal of Pure and Applied Algebra}, 212(3):500--509, 2008.

\bibitem{lpasocle}
G.~{Aranda Pino}, D.~{Martín Barquero}, C.~{Martín González}, and M.~{Siles
  Molina}.
\newblock The socle of a {L}eavitt path algebra.
\newblock {\em Journal of Pure and Applied Algebra}, 212(3):500--509, 2008.

\bibitem{BENSON2023459}
D.~J. Benson.
\newblock The socle of the group algebra of a finite p-group.
\newblock {\em Journal of Algebra}, 635:459--462, 2023.

\bibitem{kpsoc}
J.~H. Brown and A.~an~Huef.
\newblock The socle and semisimplicity of a {K}umjian–{P}ask algebra.
\newblock {\em Communications in Algebra}, 43(7):2703--2723, 2015.

\bibitem{ECP}
L.~O. Clark, R.~Exel, and E.~Pardo.
\newblock A generalized uniqueness theorem and the graded ideal structure of
  {S}teinberg algebras.
\newblock {\em Forum Math.}, 30(3):533--552, 2018.

\bibitem{CFST}
L.~O. Clark, C.~Farthing, A.~Sims, and M.~Tomforde.
\newblock A groupoid generalisation of {L}eavitt path algebras.
\newblock {\em Semigroup Forum}, 89(3):501--517, 2014.

\bibitem{CP}
Lisa~O. Clark and Yosafat E.~P. Pangalela.
\newblock Kumjian-{P}ask algebras of finitely aligned higher-rank graphs.
\newblock {\em J. Algebra}, 482:364--397, 2017.

\bibitem{Exel}
R.~Exel.
\newblock Reconstructing a totally disconnected groupoid from its ample
  semigroup.
\newblock {\em Proc. Amer. Math. Soc.}, 138(8):2991--3001, 2010.

\bibitem{farkas1981group}
D.~R. Farkas and R.~L. Snider.
\newblock Group rings which have a non-zero socle.
\newblock {\em Bulletin of the London Mathematical Society}, 13(5):392--396,
  1981.

\bibitem{jacobson}
N.~Jacobson.
\newblock {\em Structure of rings}, volume~37.
\newblock American Mathematical Soc., 1956.

\bibitem{Paterson}
A.~L.~T. Paterson.
\newblock {\em Groupoids, inverse semigroups, and their operator algebras},
  volume 170 of {\em Progress in Mathematics}.
\newblock Birkh\"auser Boston, Inc., Boston, MA, 1999.

\bibitem{Renault}
J.~Renault.
\newblock {\em A groupoid approach to {$C^{\ast} $}-algebras}, volume 793 of
  {\em Lecture Notes in Mathematics}.
\newblock Springer, Berlin, 1980.

\bibitem{richardson1976group}
J.~S. Richardson.
\newblock {\em Group rings with non-zero socle}.
\newblock PhD thesis, University of Warwick, 1976.

\bibitem{Simon1}
S.~W. Rigby.
\newblock The groupoid approach to {L}eavitt path algebras.
\newblock In {\em Leavitt path algebras and classical {$K$}-theory}, Indian
  Stat. Inst. Ser., pages 21--72. Springer, Singapore, 2020.

\bibitem{Steinberg1}
B.~Steinberg.
\newblock A groupoid approach to discrete inverse semigroup algebras.
\newblock {\em Adv. Math.}, 223(2):689--727, 2010.

\bibitem{Steinberg2}
B.~Steinberg.
\newblock Prime \'{e}tale groupoid algebras with applications to inverse
  semigroup and {L}eavitt path algebras.
\newblock {\em J. Pure Appl. Algebra}, 223(6):2474--2488, 2019.

\end{thebibliography}
\end{document}